\documentclass{birkjour}
 \newtheorem{thm}{Theorem}
 
 \newtheorem{lem}[thm]{Lemma}
 
 \theoremstyle{definition}
 \newtheorem{defn}[thm]{Definition}
 \theoremstyle{remark}
 \newtheorem{rem}[thm]{Remark}
 
\begin{document}

%
%
%
%
%
%
%
%
%

\title[Geodesic Mappings and Einstein Spaces]
 {Geodesic Mappings and Einstein Spaces}

\author[Irena Hinterleitner]{Irena Hinterleitner}

\address{%
Dept. of Math., Brno University of Technology,  \\
 \v Zi\v zkova 17,  CZ 602 00 Brno, Czech Republic}
\email{hinterleitner.i@fce.vutbr.cz}

\thanks{The paper was supported by grant P201/11/0356 of The Czech Science Foundation, MSM 6198959214, MSM 0021630511 of the Council of the Czech Government, and by the project FAST-S-11-47 of the Brno University of Technology. }
\author{Josef Mike\v s}
\address{%
Dept. of Algebra and Geometry, Palacky University, \\
17. listopadu 12, CZ 77146 Olomouc, Czech Republic}
\email{josef.mikes@upol.cz}



\date{November 13, 2011}

\begin{abstract}
In this paper we study fundamental properties of geodesic mappings with respect to the smoothness class of metrics.  
 We show that geodesic mappings preserve the smoothness class of metrics. We study geodesic mappings of Einstein spaces.
\end{abstract}
\maketitle
\def\const{{\rm const}}
\def\S{\hbox{$S$}}
\def\vn{\hbox{$V_n$}}
\def\vnn{\hbox{$\bar V_n$}}
\def\vnb{\hbox{$V_n(B)$}}
\def\vnnb{\hbox{$\bar V_n(\bar B)$}}
\def\a{\alpha}
\def\b{\beta}
\def\d{\delta}
\def\ep{\varepsilon}
\def\g{\gamma}
\def\la{\lambda}
\def\r{\varrho}
\def\t{\tau}
\def\R{\hbox{$\mathbb R$}}
\def\pod#1#2{\mathrel{\mathop{#2}\limits_{#1}}\strut}
\def\ds{\displaystyle}
\def\noi{\noindent}
\def\e{{\rm \,e\,}}
\def\tran{; transl. from }
\section{Introduction}

First we study the general dependence of geodesic mappings of (pseudo-) Riemannian manifolds in dependence on the smoothness class of the metric. We present well known facts, which were proved by Beltrami, Levi-Civita, Weyl, Sinyukov, etc., see \cite{ei,mi,mvh,rmg,si}.
In these results  no details about the smoothness class of the metric were stressed. They were formulated ``for sufficiently smooth'' geometric objects.

In the last section we present proofs of some facts about geodesic mappings of Einstein spaces.

\section{Geodesic mapping theory for \vn\ $\to$ \vnn\ of class $C^1$}

Assume the (pseudo-) Riemannian manifolds $V_n=(M,g)$ and $\bar V_n=(\bar M,\bar g)$ with metrics $g$ and $\bar g$, and Levi-Civita connections $\nabla$ and $\bar \nabla$, respectively. Here \vn, \vnn\ $\in C^1$, i.e. $g, \bar g\in C^1$ which means that their components $g_{ij}$, $\bar g_{ij}\in C^1$.
\begin{defn}
A diffeomorphism $f$: $V_n\to \bar V_n$ is called a \textit{geodesic mapping} of $V_n$ onto $\bar V_n$ if $f$ maps any geodesic  in $V_n$ onto a geodesic in $\bar V_n$.
\end{defn}

Let there exist a geodesic mapping $f$: $V_n\to \bar V_n$. Since $f$ is a diffeomorphism, we can suppose local coordinate maps on $M$ or $\bar M$, respectively, such that locally, $f$: $V_n\to \bar V_n$ maps points onto points with the same coordinates, and $\bar M=M$.
A manifold $V_n$ admits a geodesic mapping onto $\bar V_n$ if and only if the \textit{Levi-Civita equations} \begin{equation}\label{le}
\bar\nabla_XY=\nabla_XY+\psi(X)Y+\psi(Y)X
\end{equation}
hold for any tangent fields $X,Y$ and where $\psi$ is a differential form. If $\psi\equiv 0$ than $f$ is {\it affine} or {\it  trivially geodesic}.

In local form:
$
\bar\Gamma^h_{ij}=\Gamma_{ij}^h+\psi_i\delta^h_j+\psi_j\delta_i^h, 
$
\ where $\Gamma^h_{ij}(\bar\Gamma^h_{ij})$ are the Christoffel symbols of $V_n$ and $\bar V_n$, $\psi_i$ are components of $\psi$ and $\delta^h_i$ is the Kronecker delta.
Equations \eqref{le} are equivalent to the following  equations
  \begin{equation}\label{lev}
 \bar g_{ij,k}=2\psi_k \bar g_{ij}+\psi_i\bar g_{jk}+\psi\bar g_{ik}
  \end{equation}
where    ``\,,\,''  denotes the covariant derivative on $V_n$.  
It is known that  
$$\ds\psi_i=\partial_i\Psi, \quad \Psi=\frac{1}{2(n+1)}\ln\left|\frac{\det\bar g}{\det g}\right|, \quad \partial_i=\partial/\partial x^i.$$
 
Sinyukov \cite{si} proved that the Levi-Civita equations are
equivalent to
  \begin{equation}\label{si1}
  a_{ij,k}=\lambda_i g_{jk}+\lambda_j g_{ik},
  \end{equation}
  where \\[-6mm]
\begin{equation}\label{si11} 
\hbox{(a) \ } a_{ij}=\e^{2\Psi}\bar g^{\alpha\beta}g_{\alpha i}g_{\beta j};\quad 
\hbox{(b) \ } \lambda_i=-\e^{2\Psi}\bar g^{\alpha\beta}g_{\beta i}\psi_{\alpha}.
\end{equation}
 From \eqref{si1} follows
 $\lambda_i=\partial_i\la=\partial_i(\frac12\, a_{\a\b}g^{\a\b})$.  On
 the other hand \cite[p.~63]{rmg}:
\begin{equation}\label{rmg}
\bar g_{ij}=\e^{2\Psi}\tilde g_{ij},\quad
  \ds\Psi=\frac12\ln\left|\frac{\det\tilde g}{\det g}\right|,
  \quad
  \|\tilde g_{ij}\|=\|g^{i\a}g^{j\b}a_{\a\b}\|^{-1}.
\end{equation}

The above formulas are the criterion for geodesic mappings $V_n\to
\bar V_n$ globally as well as locally.

\section{Geodesic mapping theory for \vn\ $\to$ \vnn\ of class $C^2$}
Let \vn\ and \vnn\ $\in C^2$, then for geodesic mappings $V_n\to \bar
V_n$ the Riemann and the Ricci tensors transform in this way
  \begin{equation}\label{rr}
  \hbox{(a)} \quad   \bar R^h_{ijk}=R^h_{ijk}+\delta^h_k\psi_{ij}-\delta^h_j\psi_{ik};\quad
  \hbox{(b)} \quad \bar R_{ij}=R_{ij}-(n-1)\psi_{ij},
  \end{equation}
  where $ \psi_{ij}=\psi_{i,j}-\psi_i \psi_j $, 
  and the Weyl tensor of projective curvature, which is defined in the
  following form $
  W^h_{ijk}=R^h_{ijk}+\frac{1}{n-1}\left(\delta^h_kR_{ij}-\delta^h_jR_{ik}
  \right) $, 
  is invariant.
  
  The integrability conditions of the Sinyukov equations \eqref{si1}
  have the following form
\begin{equation}\label{siin}
a_{i\alpha}R^{\alpha}_{jkl}+a_{j\alpha}R^{\alpha}_{ikl}=g_{ik}\lambda_{j,l}+g_{jk}\lambda_{i,l}-g_{il}\lambda_{j,k}-g_{jl}\lambda_{i,k}.
\end{equation}
After contraction with $g^{jk}$ we get \cite{si}
\begin{equation}\label{si2}
n\lambda_{i,l}=\mu g_{il}-a_{i\alpha}R^{\alpha}_l+a_{\alpha\beta}{R^{\alpha}{}_{il}}^{\beta}
\end{equation} 
where ${R^{\alpha}{}_{il}}^{\beta}=g^{\beta k}R^{\alpha}{}_{ilk}$;
$R^{\alpha}_l=g^{\alpha j}R_{jl}$ and
$\mu=\lambda_{\a,\b}g^{\a\b}$.
  

\section{Geodesic mapping between $V_n\in C^r\ (r>2)$ and \vnn\ $\in C^2$}
\begin{thm}\label{th1} 
If \vn\ $\in C^r$ $(r>2)$ admits geodesic mappings onto \vnn\ $\in C^2$, then \vnn\ $\in C^r$.
\end{thm}
The proof of this Theorem follows from the following lemmas.
\begin{lem}\label{le2}
Let $\la^h\in C^1$ be a vector field and $\r$ a function.\\
If $\partial_i\la^h-\r\,\d^h_i\in C^1$ then $\la^h\in C^2$ and $\r\in C^1$.
\end{lem}

\begin{proof} The condition $\partial_i\la^h-\r\,\d^h_i\in C^1$  can be  written  in the following form
\begin{equation}\label{a1}
\partial_i\la^h-\r\d^h_i=f^h_i(x),
\end{equation}
where $f^h_i(x)$ are functions of class $C^1$. Evidently,  $\r\in C^0$. 
For fixed but arbitrary indices $h\neq i$ we integrate \eqref{a1} with respect to $dx^i$:
$$
\la^h=\Lambda^h + \int_{x^i_o}^{x^i}f^h_i(x^1,\dots,x^{i-1},t,x^{i+1},\dots,x^n)\,dt, 
$$
where $\Lambda^h$ is a function, which does not depend on  $x^i$. 

Because of the existence of the partial derivatives of the functions $\la^h$ and the above integrals (see \cite[p.~300]{kud}), also the derivatives $\partial_h\Lambda^h$ exist;
in this proof we don't use Einstein's summation convention.
Then we can write \eqref{a1} for $h=i$:
\begin{equation}\label{a2}
  \r=- f^h_h+\partial_h\Lambda^h  +\int_{x^i_o}^{x^i}\partial_hf^h_i(x^1,\dots,x^{i-1},t,x^{i+1},\dots,x^n)\,dt .
\end{equation}
Because the derivative with respect to $x^i$ of the right-hand side of
\eqref{a2} exists, the derivative of the function $\r$ exists, too.
Obviously $
\partial_i\r=\partial_h f^h_i - \partial_i f^h_h
$,
therefore $\r\in C^1$ and from \eqref{a1} follows $\lambda^h\in C^2$.
\end{proof}
In a similar way we can prove the following:
{\it if  $\la^h\in C^r$  $(r\geq1)$  and  $\partial_i\la^h-\r\d^h_i\in C^r$ then $\la^h\in C^{r+1}$ and $\r\in C^r$.}

\begin{lem}\label{le3}
If $V_n\!\in\! C^3$ admits a geodesic mapping onto $\bar V_n\!\in\! C^2$, then $\bar V_n\!\in\! C^3$.
\end{lem}
\begin{proof} In this case Sinyukov's equations \eqref{si1} and
  \eqref{si2} hold. According to the assumptions $g_{ij}\in C^3$ and
  $\bar g_{ij}\in C^2$.  By a simple check-up we find $\Psi\in C^2$,
  $\psi_i\in C^1$, $a_{ij}\in C^2$, $\la_i\in C^1$ and
  $R^h_{ijk},R^h{}_{ij}{}^k, R_{ij},R^h_i\in C^1$.

  From the above-mentioned conditions we easily convince ourselves
  that we can write equation \eqref{si2} in the form \eqref{a1}, where
  $\la^h=g^{h\a}\la_\a\in C^1$, $\r=\mu/n$ and
  $f^h_i=(-\la^\a\Gamma_{\a i}^h -g^{h\g}a_{\a\g}R^\a_i
   +g^{h\g}a_{\alpha\beta}{R^{\alpha}{}_{i\g}}^{\beta})/n\in C^1$.

  From Lemma \ref{le2} follows that $\la^h\in C^2$, $\r\in C^1$, and
  evidently $\la_i\in C^2$. Differentiating \eqref{si1} twice we
  convince ourselves that $a_{ij}\in C^3$.  From this and formula
  \eqref{rmg} follows that also $\Psi\in C^3$ and $\bar g_{ij}\in
  C^3$.
\end{proof} 

Further we notice that for geodesic mappings between \vn\ and \vnn\ of
class $C^3$ holds the third set of Sinyukov equations:
\begin{equation}\label{si3}
(n-1)\mu_{,k}=2(n+1)\la_\a R^\a_k+a_{\a\b}(2R^\a{}_{k,}{}^\b - 
R^{\a\b}{}_{,k}).
\end{equation}

If \vn\ $\in C^r$ and \vnn\ $\in C^2$, then by Lemma \ref{le3}, \vnn\
$\in C^3$ and \eqref{si3} hold. Because Sinyukov's system \eqref{si1},
\eqref{si2} and \eqref{si3} is closed, we can differentiate equations
\eqref{si1} $(r-1)$ times. So we convince ourselves that $a_{ij}\in
C^r$, and also $\bar g_{ij}\in C^r$ $(\equiv\bar V_n\in C^r)$.
\endproof

\begin{rem}
  Because for holomorphically projective mappings of K\"ahler (and
  also hyperbolic and parabolic K\"ahler) spaces hold equations
  analogical to \eqref{si1} and \eqref{si2}, see \cite{miho,mvh,si},
  from Lemma \ref{le2} follows an analog to Theorem \ref{th1} for
  these mappings.

\end{rem}

\section{On geodesic mappings of Einstein spaces}

Geodesic mappings of Einstein spaces were studied by many authors
starting with A.Z.~Petrov (see \cite{pe}).  Einstein spaces \vn\ are
characterized by the condition $Ric=\const\cdot g$, so \vn $\in C^2$
would be sufficient. But many properties of Einstein spaces appear when
$\vn\in C^3$ and $n>3$. An Einstein space $V_3$ is a space of constant
curvature. 

We continue with geodesic mappings of Einstein spaces $\vn\in C^3$. On
basis of Theorem \ref{th1} it is natural to suppose that \vnn $\in
C^3$. In 1978 (see PhD thesis \cite{mid} and \cite{miei}) Mike\v s
proved that under these conditions the following theorem holds:
\begin{thm}\label{thei}
  If the Einstein space $V_n$ admits a nontrivial geodesic mapping
  onto a (pseudo-) Riemannian space $\bar V_n$, then $\bar V_n$ is an
  Einstein space.
\end{thm}
%
%
\begin{proof}
  Let the Einstein space $V_n\in C^3$ (for which
  $R_{ij}=-K\,(n-1)\,g_{ij}$) admit a nontrivial geodesic mapping onto
  $\bar V_n\in C^2$.  Then the Sinyukov equations \eqref{si1} hold;
  their integrability conditions have the  form
  \eqref{siin}. Taking \eqref{si1} into account, we differentiate
  \eqref{siin} with respect to $x^m$, contract the result with
  $g^{lm}$, and then we alternate with respect to $i$, $k$. By
  \eqref{a1}, we get $\la_\a R^\a_{ijk}=g_{ij}\xi_k-g_{ik}\xi_j$,
  where $\xi_i$ is some vector. Contracting the latter with $g^{ij}$
  and using \eqref{a1} we see that $\xi_i=K\la_i$, that is, the
  formula reads $\la_\a R^\a_{ijk}=K(g_{ij}\la_k-g_{ik}\la_j)$.

  We contract \eqref{siin} with $\la^l$. Considering the last formula,
  we get
\begin{equation}\label{eee}
g_{ki}\Lambda_{j\a}\la^\a +  g_{kj}\Lambda_{i\a}\la^\a -
\la_i\Lambda_{jk} - \la_j\Lambda_{ik}=0,
\end{equation}
where $\Lambda_{ij}=\lambda_{i,j}-Ka_{ij}$. It is easy to show that
$\la^\a \Lambda_{\a i}=\mu\la_i$, where $\mu$ is a function. Since
$\la_i\neq0$, we find from \eqref{eee} that
\begin{equation}\label{ei1}
  \la_{i,j}=\mu\,g_{ij}+K\,a_{ij}.
\end{equation}
Differentiating (\ref{si11}b) and considering \eqref{lev}, \eqref{si1}, \eqref{si11}, it
is easy to get the following equation:
\begin{equation}\label{f10}
\psi_{ij}\equiv\psi_{i,j}-\psi_i\psi_j=\bar K\,g_{ij}-K\bar g_{ij},
\end{equation}  
where $\bar K$ is a function. Then from (\ref{rr}b), by virtue of the last
relation, and considering $R_{ij}=-K\,(n-1)\,g_{ij}$, we get that
$\bar R_{ij}=(n-1)\bar K\,\bar g_{ij}$. Hence \vnn\ is an Einstein
space. The theorem is proved.
\end{proof}

Theorem \ref{thei} was proved ``locally'' but it is easy to show that
when the domain of validity of equations \eqref{f10} borders with a
domain where $\psi_i\equiv0$, then in this domain
$\psi_i\equiv0$. Assume a point $x_0$ on the borders between these
domains, then $\psi_i(x_0)=0$ and $\psi_{ij}=0$. Indeed {\bf a)} If
$K\neq0$ or $\bar K\neq0$ then $\bar g_{ij}(x_0)=\bar K/K\
g_{ij}(x_0)$. From these properties follows that the system of
equations \eqref{lev} and \eqref{f10} has a unique solution $\bar g_{ij}=\bar K/K\
g_{ij}$ and $\psi_i=0$.  {\bf b)}~If $K=\bar K=0$ then equations
\eqref{f10}: $\psi_{i,j}=\psi_i\psi_j$ have a unique solution for
$\psi_i(x_0)=0$: $\psi_i=0$.

This Theorem was used for geodesic mappings of 4-dimensional Einstein
spaces (Mike\v s, Kiosak \cite{miki82}) and to find metrics of
Einstein spaces that admit geodesic mappings (Formella, Mike\v s
\cite{formik}), etc.

\end{document}